\documentclass{tran-l}

\newtheorem{theorem}{Theorem}[section]
\newtheorem{lemma}[theorem]{Lemma}
\newtheorem{proposition}[theorem]{Proposition}
\newtheorem{corollary}[theorem]{Corollary}

\theoremstyle{definition}

\theoremstyle{remark}

\numberwithin{equation}{section}

\def\Leb{\operatorname {\rm Leb}}
\def\leb{\operatorname {\rm leb}}
\def\iu{\sqrt{-1}\,}

\begin{document}

\title[The central limit theorem for  Riesz-Raikov sums II]{The central limit theorem for  Riesz-Raikov sums II}


\author{Katusi FUKUYAMA}
\address{Department of Mathematics, \\
Kobe University, \\
Rokko, Kobe, 657-8501, Japan\\
}
\email{fukuyama@math.kobe-u.ac.jp}
\thanks{This research is partially supported by JSPS KAKENHI 16K05204 and 15KT0106.
It was also partially supported by the Research Institute for Mathematical Sciences, 
a Joint Usage/Research Center located in Kyoto University.
}


\date{}

\dedicatory{Dedicated to Professor Norio K\^ono on his 80th birthday}

\begin{abstract}
For a $d\times d$ expanding matrix $A$, we investigate 
randomness of the sequence $\{A^k \boldsymbol x\}$ and prove the 
central limit theorem for $\sum f(A^k \boldsymbol x)$
where $f$ is a periodic function with a mild regularity condition. 
\end{abstract}

\maketitle


\bibliographystyle{amsplain}


\section{Introduction}

For $\theta > 1$ and  a real valued square integrable function $f$ on $\bf R$ with period 1 satisfying 
$\int_0^1 f=0$, regarded as  random variables on $[\,0,1\,]$ 
Riesz-Raikov sums $\sum_{k=1}^N f(\theta^k x)$ obey the central limit theorem. 
This fact was first proved by Fortet \cite{fortet} and Kac \cite{kac} when $\theta$ is an integer, and 
was extended to general case by Petit \cite{petit} and by the author \cite{1994F}. 
In \cite{2015F} 
the case when $\theta$ is a complex number was investigated.  
In this note, we consider a multidimensional analogue of this problem. 

Let $\mathbf R^d$ denote the vector space of $d$-dimensional real column vectors 
and $\widehat {\mathbf R}^d$ that of  real row vectors. 

Fan \cite{fan1} assumed a very mild condition on $A$ and 
proved that the sequence $\{A^n \boldsymbol x\}$ is uniformly distributed mod 1 
for a.e. $\boldsymbol x$, and  Lesigne \cite{lesigne} extended this result. 
Since the result of Fan implies the law of large numbers for $\{f(A^n \boldsymbol x)\}$
where $f$ is periodic and continuous, 
it is very natural to ask whether the central limit theorem holds. 

To have the central limit theorem, 
we here assume that $A$ is {\it expanding}, i.e., there exists a $q> 1$ such that 
\begin{equation}
\|A \boldsymbol x\|_2 \ge q \| \boldsymbol x\|_2 
\quad\hbox{for}\quad \boldsymbol x \in  \mathbf R^d,
\label{Eq:expx}
\end{equation}
or equivalently, 
\begin{equation}
\|\boldsymbol \xi A \|_2 \ge q \| \boldsymbol \xi\|_2 
\quad\hbox{for}\quad \boldsymbol \xi \in \widehat {\mathbf R}^d .
\label{Eq:expxi}
\end{equation}
Let $f$ be a real valued square integrable function on $\mathbf R^d$ which satisfies
\begin{equation}
f(\boldsymbol x+ \boldsymbol e_i) = f(\boldsymbol x) \quad(i=1, \dots, d)
\quad\hbox{and}\quad
\int_{[\,0,1)^d} f(\boldsymbol x)\,d\boldsymbol x=0,
\label{Eq:periodmean}
\end{equation}
where $\boldsymbol e_1$, \dots, $\boldsymbol e_d$ is the canonical basis of  $\mathbf R^d$. 
We denote the Fourier series of $f$ by
$$
f(\boldsymbol x) \sim \sum_{\boldsymbol \xi \in {\bf Z}^d}\widehat f(\boldsymbol \xi) 
\exp(2\pi \iu \boldsymbol \xi \boldsymbol x),
$$
where 
 ${\bf Z}^d$ is regarded as a subset of $\widehat {\mathbf R}^d$, 
and define the subsum $f_a$ by 
\begin{equation}\label{Eq:subsum}
f_a(\boldsymbol x) = \sum_{\boldsymbol \xi \in R_a}\widehat f(\boldsymbol \xi)
\exp(2\pi \iu \boldsymbol \xi \boldsymbol x) ,
\end{equation}
where
$R_a =\{\boldsymbol \xi=(\xi_1, \dots, \xi_d)\in {\bf Z}^d\setminus\{\boldsymbol 0\} 
\mid |\xi_1|, \dots, |\xi_d|\le a
\}$.
Note that we have $\widehat f(\boldsymbol 0)  = 0$. 
We assume that $f$ satisfies the condition
\begin{equation}\label{Eq:dini}
\sum_{m=1}^\infty \|f-f_{2^m}\|_{L^2[0,1)^d} < \infty.
\end{equation}\
Note that a function 
of bounded variation over $[\,0,1)^d$ in the sense of Hardy-Krause 
satisfies this condition (Cf. \cite{1968Z}). 
An easy sufficient condition for (\ref{Eq:dini}) is the $L^2$-Dini condition
\begin{equation}
\label{Eq:l2dini}
\int_0^1 \frac{\omega^{(2)}(H, f)}{H}\,dH<\infty,
\end{equation}
where 
$$
\omega^{(2)}(H, f)
=
\sup
\biggl\{\,\|f({}\cdot{}+\boldsymbol h)-f({}\cdot{})\|_{L^2[0,1]^d}\biggm| 
{\boldsymbol h=(h_1, \dots, h_d)\in {\bf R}^d, 
\atop
|h_1|, \dots, |h_d|\le H}\,\biggr\}
.
$$
This fact is stated in Lemma \ref{lem:l2dini}. 

To state our result, we introduce a quantity $\sigma^2(f)$, the limiting variance of
the central limit theorem, by 
\begin{equation}\label{Eq:lv}
\sigma^2(f) = 
\sum_{l\ge 0} (2-\delta_{0,l}) \sum_{\boldsymbol \xi, \boldsymbol \xi'\in {\bf Z}^d}
\widehat f(\boldsymbol \xi) \widehat f(\boldsymbol \xi')
{\mathbf 1}(\boldsymbol \xi+ \boldsymbol \xi'A^l=\boldsymbol 0)
.
\end{equation}
The series on the right hand side 
is  shown to be absolutely convergent under the conditions (\ref{Eq:expx})
and   (\ref{Eq:dini}). 
We denote the Lebesgue measure on $\mathbf R^d$ by $\Leb$. 

\begin{theorem}\label{thm:main}
Let $A$ be a $d\times d$ real matrix satisfying (\ref{Eq:expx}), 
and let $f$ be a real valued function on $\mathbf R^d$ satisfying
(\ref{Eq:periodmean}) and (\ref{Eq:dini}). 
Then for every bounded measurable  $\Gamma \subset \mathbf R^d$, 
\begin{equation}
\lim_{N\to\infty}
\int_{\Gamma}
\biggl(\frac1{\sqrt N} 
\sum_{k=1}^N 
f(A^k \boldsymbol x)
\biggr)^2 
\,d\boldsymbol x
=
\Leb(\Gamma)
\sigma^2(f),
\label{Eq:limvar}
\end{equation}
and
\begin{equation}
\lim_{N\to \infty}
\Leb
\biggl\{ \boldsymbol x \in \Gamma 
\biggm| 
\frac1{\sqrt N} 
\sum_{k=1}^N 
f(A^k \boldsymbol x)
\le t
\biggr\}
=
\Leb (\Gamma) 
\Phi_{\sigma^2(f)}(t) 
\label{Eq:CLT}
\end{equation}
for any $t\not=0$, 
where $\Phi_{\sigma^2(f)}$ is the distribution function of  $N(0, \sigma^2(f))$. 
If $\sigma^2(f)> 0$, then  (\ref{Eq:CLT}) holds also for $t=0$. 
\end{theorem}

In the case when every coefficient of $A$ is an integer, 
that is the case when $A^n$ are endomorphisms on ${\bf T}^d$, 
Leonov \cite{leonov}, Fan \cite{fan0}, 
Levin \cite{levin} and Conze, Le Borgne, and Roger \cite{conze} assumed 
so called the partially expanding condition
and proved the central limit theorem. 

L\"obbe \cite{lobbe} 
proved the central limit theorem, the law of the iterated logarithm, 
and the metric discrepancy results for $\{A_n \boldsymbol x\}$, 
where 
$A_n$ is a $d\times d$ matrix with integer coefficients satisfying
$$
\|\boldsymbol j A_{n+k}\|_\infty \ge q^k  \|A_n\|_\infty
\quad\hbox{if}\quad \boldsymbol j\in {\bf Z}^d\quad\hbox{and}\quad k\ge \log _q \|\boldsymbol j\|
$$
for some $q> 1$. 

By putting $\boldsymbol j = (1, 0)$ and $1<c_1<c_2$, 
we see that this condition cannot be satisfied even by 
$$
A_n = 
\begin{pmatrix}
c_1 & 0 
\\
0   & c_2
\end{pmatrix}^n
.
$$
This example suggests the following inference. 
For a matrix $A$,  the sequence of matrix $A_n = A^n$ satisfies this condition 
only if absolute values of characteristic values of $A$ are all equal. 

Since the condition (\ref{Eq:expx}) implies that modulus of all eigenvalues of $A$ are 
greater than one, 
we expect that the central limit theorem is valid under this assumption. 
But we could not prove because of some technical reason. 

In the case of endomorphisms on ${\bf T}^d$, 
if $\sigma(f)=0$, one can express 
$f(\boldsymbol x)= g(A\boldsymbol x)- g(\boldsymbol x)$ by using some locally square integrable 
periodic $g$. 
We could not prove it in our case, 
since the situation is more complicated. 
For example if 
$$
A = 
\begin{pmatrix}
\root 3\of 2 & 0 
\\
0   & \root 2 \of 3
\end{pmatrix}
, 
$$
we have $\sigma(f)=0$ for the function $f$ of the form
$$f(x) = g_1(A^6 \boldsymbol x)- g_1( \boldsymbol x)
+ g_2(A^3 \boldsymbol x) - g_2( \boldsymbol x)
+ g_3(A^2 \boldsymbol x) - g_3( \boldsymbol x)
$$
where $g_1$, $g_2$, and $g_3$ are locally square integrable and periodic, 
$g_2$ does not depend on $x_2$, and $g_3$ does not depend of $x_1$. 
Although by a similar argument as one-dimensional case, 
we can prove that $\sigma(f)=0$ implies this expression,
we do not have the result for general case so far.  

\section{Real Jordan form and related estimates}

In this section we state some preliminary facts. 

Irrespectively of the value of $\gamma\in \mathbf N$, 
for  vectors $\boldsymbol x={}^T(x_1, \dots, x_\gamma)\in \mathbf R^\gamma$
and $\boldsymbol \xi= (\xi_1, \dots, \xi_\gamma)\in \widehat{\mathbf R}^\gamma$, 
we define 
$\|\boldsymbol x\|= \max_{\delta\le \gamma}|x_\delta|$
and $\|\boldsymbol \xi\|= \max_{\delta\le \gamma}|\xi_\delta|$.
Abusing the notation we denote by $\boldsymbol 0$  zero vectors in 
${\mathbf R}^\gamma$ and $\widehat{\mathbf R}^\gamma$ for any $\gamma$. 

For $\lambda\in {\bf R}$ 
a standard Jordan block $J_\gamma(\lambda)$ is a $\gamma\times \gamma$ matrix,
and 
for $\lambda\in {\bf C}\setminus {\bf R}$ a real Jordan block $C_\gamma(\lambda)$ is a  $2\gamma\times 2\gamma$ matrix 
defined as follows.
$$
J_\gamma(\lambda)=
\begin{pmatrix}
\lambda  & 1 &  & \smash{\raise -4pt\hbox{\Large  0}}
\\
                    &\lambda  & \ddots &  
\\
                    &                    & \ddots & 1
\\
\smash{ \raise 0pt\hbox{\Large  0}} &         &        & \lambda
\end{pmatrix},
\quad
C_\gamma(\lambda)=
\begin{pmatrix}
|\lambda | Z_\theta & Z_0 &  & \smash{\raise -4pt\hbox{\Large \bf 0}}
\\
                    &|\lambda | Z_\theta & \ddots &  
\\
                    &                    & \ddots & Z_0
\\
\smash{ \raise 0pt\hbox{\Large \bf 0}} &         &        & |\lambda | Z_\theta 
\end{pmatrix},
$$
where 
$
Z_\theta = \begin{pmatrix} \cos \theta & -\sin \theta \\ \sin\theta & \hfill\cos \theta \end{pmatrix}
$ and  
$\theta = \arg \lambda$.

It is known (See e.g. Theorem 6.65 of \cite{shirov})
that a real matrix is similar to the matrix of the form
$$
B=
\begin{pmatrix}
J_{d_1}(\lambda_1) & & &       &        &\smash{\raise -4pt\hbox{\Large  0}}
\\
     & \ddots 
\\
     &        & J_{d_\alpha}(\lambda_\alpha) &
\\
     &        &        & C_{d_{\alpha+1}/2}(\lambda_{\alpha+1})
\\
     &        &        &       & \ddots
\\
\smash{\raise 0pt\hbox{\Large  0}} &    &    &    &    & C_{d_{\beta}/2}(\lambda_{\beta})
\end{pmatrix},
$$
where 
$\lambda_1$, \dots, $\lambda_\alpha\in \mathbf R$, 
$\lambda _{\alpha+1}$, $\bar \lambda _{\alpha+1}$, \dots , $\lambda _{\beta}$, $\bar \lambda _{\beta}
\in \mathbf C \setminus \mathbf R$
are characteristic values, 
i.e. 
there exists a real regular matrix 
$$
Q = (\boldsymbol q_1^{(1)}, \dots, \boldsymbol q_{d_1}^{(1)}, 
\boldsymbol q_1^{(2)}, \dots, \boldsymbol q_{d_2}^{(2)}, \dots, 
\boldsymbol q_1^{(\beta)}, \dots, \boldsymbol q_{d_\beta}^{(\beta)}
)$$
such that 
$A = Q B Q^{-1}$. 
We simply write
$
B = 
\begin{pmatrix}
B_1 &                                   &        \smash{\raise -4pt\hbox{\Large \bf 0}}
\\
    & \ddots 
\\
\smash{\raise 0pt\hbox{\Large \bf 0}}    &        & B_\beta
\end{pmatrix}
.$

Since $Q$ is regular, there exists a constant $1< D_{1,Q}< \infty$ such that
\begin{align}
& \frac{\| Q \boldsymbol x \|}{\|  \boldsymbol x \|}, 
\frac{\| Q^{-1} \boldsymbol x \|}{\|  \boldsymbol x \|}, 
\frac{\| \boldsymbol \xi Q  \|}{\|  \boldsymbol \xi \|}, 
\frac{\| \boldsymbol \xi Q^{-1}  \|}{\|  \boldsymbol \xi \| }
\in (D_{1,Q}^{-1}, D_{1,Q})
\quad(\boldsymbol x \not=\boldsymbol 0, 
\boldsymbol \xi \not=\boldsymbol 0).
\label{Eq:xQest}
\end{align}

Denoting the span of  $\boldsymbol q_1^{(h)}$, \dots , $\boldsymbol q_{d_h}^{(h)}$
by $W_h$, 
we have 
$$
\mathbf R^d = \bigoplus _{h=1}^\beta W_h, 
\quad \dim W_h  = d_h
\quad\hbox{and}
\quad \sum_{h=1}^\beta d_h = d
.$$

By calculating $J_\gamma(\lambda)^k$ and $C_\gamma(\lambda)^k$ for $|\lambda|>1$, we have
$\|J_\gamma(\lambda)^k \boldsymbol x\| \le \gamma k^\gamma |\lambda|^k \|\boldsymbol x\|$
and 
$\|C_\gamma(\lambda)^k \boldsymbol x\| \le 2\gamma k^\gamma |\lambda|^k \|\boldsymbol x\|$.
Hence there exists a constant $D_{2,A}$ depending only on $A$ such that 
\begin{equation}
\|A^k \boldsymbol x_h\| \le  D_{2,A} |\lambda_h|^k k^d \|\boldsymbol x_h\| 
\quad(\boldsymbol x_h\in W_h).
\label{Eq:upperb}
\end{equation}

Denote the $\delta$-th component of $\boldsymbol \xi$ by $(\boldsymbol \xi)_\delta$, 
and the 2-dimensional vector consisting of the $\delta$-th and the $(\delta+1)$-th components 
of $\boldsymbol \xi$ by $(\boldsymbol \xi)_{\delta,\delta+1}$. 

Put 
$m_1(\boldsymbol 0)=0$ and 
 $m_1(\boldsymbol \xi)= |\xi_\delta|$ if 
 $\boldsymbol \xi = (\xi_1, \dots, \xi_\gamma)\in \widehat{\mathbf R}^\gamma$
satisfies $\xi_1= \cdots = \xi_{\delta-1}=0 \not= \xi_\delta$.
In the last case, 
we have
$
(\boldsymbol \xi J_\gamma(\lambda)^k )_1= \cdots 
= (\boldsymbol \xi J_\gamma(\lambda)^k )_{\delta-1}= 0$, 
$(\boldsymbol \xi J_\gamma(\lambda)^k )_{\delta}= \lambda ^k\xi_\delta$, 
or $m_1(\boldsymbol \xi J_\gamma(\lambda)^k )= |\lambda| ^km_1(\boldsymbol \xi )$. 

Put $m_2(\boldsymbol 0)=0$ and 
$m_2(\boldsymbol \xi)= \|\boldsymbol \eta_\delta\|_2$ 
if $\boldsymbol \xi = (\boldsymbol \eta_1, \dots, \boldsymbol \eta_\gamma)\in 
\widehat{\mathbf R}^{2\gamma}$ 
($\boldsymbol \eta_1$, \dots, $\boldsymbol \eta_\gamma\in \widehat{\mathbf R}^2$)
satisfies
 $\boldsymbol \eta_1 = \cdots = \boldsymbol \eta_{\delta-1}=\boldsymbol 0\not=\boldsymbol \eta_\delta$. 
In the last case 
we have
$(\boldsymbol \xi C_\gamma(\lambda)^k)_{1,2}= \cdots
=(\boldsymbol \xi C_\gamma(\lambda)^k)_{2\delta-3,2\delta-2}= \boldsymbol 0$ and 
$
(\boldsymbol \xi C_\gamma(\lambda)^k)_{2\delta-1,2\delta}
=
|\lambda|^k \boldsymbol \eta_\delta Z_{k\theta}
$,  or 
$m_2(\boldsymbol \xi C_\gamma(\lambda)^k)= |\lambda|^k m_2(\boldsymbol \xi)$. 

For $\boldsymbol \xi\in {\bf Z}^d\setminus\{\boldsymbol 0\}$, 
we write $\boldsymbol \xi Q = (\widetilde{\boldsymbol \xi}_1, \dots, \widetilde{\boldsymbol \xi}_\beta)$
where $\widetilde {\boldsymbol \xi}_h \in \widehat {\mathbf R}^{d_h}$. 
Put $i(h) = 1$ for $h\in [\,1,\alpha\,]$ and $i(h) = 2$ for $h\in [\,\alpha+1, \beta\,]$. 
The  arguments above show that 
$
m_{i(h)}
(\widetilde{\boldsymbol \xi}_{h} B_{h}^k )
=  |\lambda_{h}|^k
m_{i(h)}
(\widetilde{\boldsymbol \xi}_{h})
$.

We denote $\Lambda =\min_{h\le\beta}|\lambda_h|> 1$. 
Since $Q$ is regular, 
there exists an $h\le \beta$ such that $\widetilde{\boldsymbol \xi}_h 
\not = \boldsymbol 0$. 
We denote the smallest such $h$ by $h({\boldsymbol \xi})$. 
Denote 
$
D_{3,a} 
=
\min_{\boldsymbol \xi\in R_a}
m_{i(h({\boldsymbol \xi}))}(\widetilde {\boldsymbol \xi}_{h({\boldsymbol \xi})} )
$. 
The arguments above show that 
\begin{equation}
m_{i(h({\boldsymbol \xi}))}
(\widetilde{\boldsymbol \xi}_{h({\boldsymbol \xi})} B_{h({\boldsymbol \xi})}^k )
\ge 
D_{3,a}  |\lambda_{h({\boldsymbol \xi})}|^k
\ge  D_{3,a}\Lambda^k
\quad(\boldsymbol \xi \in R_a)
.
\label{Eq:lowerb}
\end{equation}

We put
\begin{align*}
D_{4,a}(K)&=
\min
\biggl\{
m_{i(h({\boldsymbol \xi}))}(\widetilde {\boldsymbol \xi}_{h({\boldsymbol \xi})} 
B_{h({\boldsymbol \xi})}^k + \widetilde {\boldsymbol \xi}'_{h({\boldsymbol \xi})})
\biggm|
{0\le k< K,\ \boldsymbol \xi, \boldsymbol \xi'\in R_a, \hfill\atop
\widetilde {\boldsymbol \xi}_{h({\boldsymbol \xi})} 
B_{h({\boldsymbol \xi})}^k + \widetilde {\boldsymbol \xi}'_{h({\boldsymbol \xi})}
\not={\bf 0}
}
\biggr\},
\end{align*}
and prove that $D_{4,a}:=D_{4,a}(\infty)$ is positive. Put 
$\ell_1 = \max\{ m_{i(h)}(\widetilde {\boldsymbol \xi}'_h) \mid \boldsymbol \xi'\in R_a, h\le\beta\}$ 
and
$\ell_2 = \min\{ m_{i(h)}(\widetilde {\boldsymbol \xi}'_h) \mid \boldsymbol \xi'\in R_a, h\le\beta, 
m_{i(h)}(\widetilde {\boldsymbol \xi}'_h) \not=0\}> 0$. 
By (\ref{Eq:lowerb}), 
there exists a $K_0$ such that $k\ge K_0$ implies 
$m_{i(h({\boldsymbol \xi}))}
(\widetilde{\boldsymbol \xi}_{h({\boldsymbol \xi})} B_{h({\boldsymbol \xi})}^k )
\ge 2\ell_1$. 

For  $k\ge K_0$, we can verify
$m_{i(h({\boldsymbol \xi}))}
(\widetilde{\boldsymbol \xi}_{h({\boldsymbol \xi})} B_{h({\boldsymbol \xi})}^k + 
\widetilde{\boldsymbol \xi}_{h({\boldsymbol \xi})}')
\ge \ell_1\wedge \ell_2$. 
Here we prove it in the case $h({\boldsymbol \xi})=1$. 
The other case can be proved in the same way. 
Suppose that
$(\widetilde {\boldsymbol \xi}_{h({\boldsymbol \xi})}B_{h({\boldsymbol \xi})}^k)_1=
\cdots = (\widetilde {\boldsymbol \xi}_{h({\boldsymbol \xi})}B_{h({\boldsymbol \xi})}^k)_{\delta-1}=0$, 
 $|(\widetilde {\boldsymbol \xi}_{h({\boldsymbol \xi})}B_{h({\boldsymbol \xi})}^k)_{\delta}|=
m_{1}(\widetilde {\boldsymbol \xi}_{h({\boldsymbol \xi})} 
B_{h({\boldsymbol \xi})}^k )
$.
If $\widetilde{\boldsymbol \xi}_{h({\boldsymbol \xi})}'={\bf 0}$, 
we have
$\widetilde {\boldsymbol \xi}_{h({\boldsymbol \xi})} 
B_{h({\boldsymbol \xi})}^k + \widetilde {\boldsymbol \xi}'_{h({\boldsymbol \xi})}
=\widetilde {\boldsymbol \xi}_{h({\boldsymbol \xi})} 
B_{h({\boldsymbol \xi})}^k
$ and 
$
m_{1}(\widetilde {\boldsymbol \xi}_{h({\boldsymbol \xi})} 
B_{h({\boldsymbol \xi})}^k + \widetilde {\boldsymbol \xi}'_{h({\boldsymbol \xi})})
=
m_{1}(\widetilde {\boldsymbol \xi}_{h({\boldsymbol \xi})} 
B_{h({\boldsymbol \xi})}^k)
\ge 2\ell _1$. 
If not, suppose that
$(\widetilde {\boldsymbol \xi}'_{h({\boldsymbol \xi})})_1 =
\dots =(\widetilde {\boldsymbol \xi}'_{h({\boldsymbol \xi})})_{\delta'-1} =0$ 
and $|(\widetilde {\boldsymbol \xi}'_{h({\boldsymbol \xi})})_{\delta'}| =
m_1(\widetilde {\boldsymbol \xi}'_{h({\boldsymbol \xi})})_{\delta'} $. 
If $\delta'< \delta$, 
then we have
$(\widetilde {\boldsymbol \xi}_{h({\boldsymbol \xi})} 
B_{h({\boldsymbol \xi})}^k + \widetilde {\boldsymbol \xi}'_{h({\boldsymbol \xi})})_1
=
\cdots
=
(\widetilde {\boldsymbol \xi}_{h({\boldsymbol \xi})} 
B_{h({\boldsymbol \xi})}^k + \widetilde {\boldsymbol \xi}'_{h({\boldsymbol \xi})})_{\delta'-1}
=0
$
and 
$
|(\widetilde {\boldsymbol \xi}_{h({\boldsymbol \xi})} 
B_{h({\boldsymbol \xi})}^k + \widetilde {\boldsymbol \xi}'_{h({\boldsymbol \xi})})_{\delta'}|
=
|(\widetilde {\boldsymbol \xi}'_{h({\boldsymbol \xi})})_{\delta'}|
=
m_1(\widetilde {\boldsymbol \xi}'_{h({\boldsymbol \xi})})_{\delta'}
\ge \ell_2
$. 
Hence we have
$m_1
(\widetilde {\boldsymbol \xi}_{h({\boldsymbol \xi})} 
B_{h({\boldsymbol \xi})}^k + \widetilde {\boldsymbol \xi}'_{h({\boldsymbol \xi})})
\ge \ell_2$
in this case. 
If $\delta'= \delta$, 
then we have
$(\widetilde {\boldsymbol \xi}_{h({\boldsymbol \xi})} 
B_{h({\boldsymbol \xi})}^k + \widetilde {\boldsymbol \xi}'_{h({\boldsymbol \xi})})_1
=
\cdots
=
(\widetilde {\boldsymbol \xi}_{h({\boldsymbol \xi})} 
B_{h({\boldsymbol \xi})}^k + \widetilde {\boldsymbol \xi}'_{h({\boldsymbol \xi})})_{\delta-1}
=0
$
and 
$
|(\widetilde {\boldsymbol \xi}_{h({\boldsymbol \xi})} 
B_{h({\boldsymbol \xi})}^k + \widetilde {\boldsymbol \xi}'_{h({\boldsymbol \xi})})_{\delta}|
\ge
|(\widetilde {\boldsymbol \xi}_{h({\boldsymbol \xi})} 
B_{h({\boldsymbol \xi})}^k )_{\delta}|
-
|(\widetilde {\boldsymbol \xi}'_{h({\boldsymbol \xi})})_{\delta}|
\ge 2\ell_1-\ell_1
$. 
Hence we have
$m_1
(\widetilde {\boldsymbol \xi}_{h({\boldsymbol \xi})} 
B_{h({\boldsymbol \xi})}^k + \widetilde {\boldsymbol \xi}'_{h({\boldsymbol \xi})})
\ge \ell_1$
in this case. 
If $\delta'>  \delta$, 
then we have
$(\widetilde {\boldsymbol \xi}_{h({\boldsymbol \xi})} 
B_{h({\boldsymbol \xi})}^k + \widetilde {\boldsymbol \xi}'_{h({\boldsymbol \xi})})_1
=
\cdots
=
(\widetilde {\boldsymbol \xi}_{h({\boldsymbol \xi})} 
B_{h({\boldsymbol \xi})}^k + \widetilde {\boldsymbol \xi}'_{h({\boldsymbol \xi})})_{\delta-1}
=0
$
and 
$
|(\widetilde {\boldsymbol \xi}_{h({\boldsymbol \xi})} 
B_{h({\boldsymbol \xi})}^k + \widetilde {\boldsymbol \xi}'_{h({\boldsymbol \xi})})_{\delta}|
=
|(\widetilde {\boldsymbol \xi}_{h({\boldsymbol \xi})} 
B_{h({\boldsymbol \xi})}^k )_{\delta}|
\ge 2\ell_1
$. 
Hence we have
$m_1
(\widetilde {\boldsymbol \xi}_{h({\boldsymbol \xi})} 
B_{h({\boldsymbol \xi})}^k + \widetilde {\boldsymbol \xi}'_{h({\boldsymbol \xi})})
\ge 2\ell_1$
in this case.

Since $D_{4,a}(K_0)$ 
is positive, 
we have  $D_{4,a}> 0$. 

\begin{lemma}\label{lem:l2dini}
The condition (\ref{Eq:l2dini}) implies the condition (\ref{Eq:dini}). 
\end{lemma}
\begin{proof}
Let $\|\boldsymbol h\|\le H$. 
By 
$$
f(\boldsymbol x + \boldsymbol h)
- 
f(\boldsymbol x)
=
\sum_{\boldsymbol \xi\in {\bf Z}^d}
\widehat f(\boldsymbol \xi) 
\exp(2\pi \iu\boldsymbol \xi \boldsymbol x)
\bigl(
\exp(2\pi \iu\boldsymbol \xi \boldsymbol h)
-1\bigr),
$$
we have
\begin{align*}
\bigl(\omega^{(2)}(H,f)\bigr)^2
&\ge 
\|f({}\cdot{}+\boldsymbol h) 
-
f({}\cdot{})
\|_{L^2[0,1]^d}^2
\\&
=
2\sum_{\boldsymbol \xi\in {\bf Z}^d}
|\widehat f(\boldsymbol \xi) |^2
\bigl(1- {\rm Re}
\exp(2\pi \iu\boldsymbol \xi \boldsymbol h)
\bigr).
\end{align*}
By integrating both sides over $[\,-H, H\,]^d$ by $\boldsymbol h$ and dividing by $(2H)^d$, we have
$$
\bigl(\omega^{(2)}(H,f)\bigr)^2
\ge
2\sum_{\boldsymbol \xi\in {\bf Z}^d}
|\widehat f(\boldsymbol \xi) |^2
\biggl(1-\prod_{i=1}^d\frac{\sin(2\pi \xi_i H)}{2\pi \xi_i H}\biggr)
.
$$
If $\|\boldsymbol \xi\|\ge 1/H$, then there exists a $j_0$ such that $|\xi_{j_0}|H\ge 1$. 
Since we have
$$
\biggl|\prod_{i=1}^d\frac{\sin(2\pi \xi_i H)}{2\pi \xi_i H}\biggr|
\le \frac1{2\pi |\xi_{j_0}|H}\le \frac1{2\pi}
,$$
we have
$$
\bigl(\omega^{(2)}(H,f)\bigr)^2
\ge
(2-1/\pi)\sum_{\boldsymbol \xi\in {\bf Z}^d: \|\boldsymbol \xi\|\ge 1/H}
|\widehat f(\boldsymbol \xi) |^2
.
$$
Hence we have $\omega^{(2)}(2^{-m},f)\ge \|f-f_{2^m}\|_{L^2[0,1]^d}$. 
Since the condition (\ref{Eq:l2dini}) is equivalent to 
$\sum_m \omega^{(2)}(2^{-m},f)< \infty$, it implies the condition (\ref{Eq:dini}). 
\end{proof}

\section{Fourth moment estimates}

In this section, we assume the condition (\ref{Eq:expx}), 
Put
$$
\rho_0(x) = \Bigl(\frac{\sin x}x\Bigr) ^2 
\quad\hbox{and}
\quad
\rho_1(x) = \rho_0\bigl(x/\sqrt d\,\bigr) + \rho_0\bigl(x/\sqrt{2d}\,\bigr)
.$$
By
$$
\widehat \rho_0(\xi)=
\int_{\bf R} \rho_0(x) e^{2\pi \iu \xi x}\,dx
= 
\pi (1-\pi |\xi|)\vee0, 
$$
we obtain
$\widehat \rho_1(\xi) =0$ if $|\xi | \ge 1/(\pi\sqrt d\,)$ and 
$\rho_1(x) > 0$ for $x\in {\bf R}$. 
By putting
$$\rho(\boldsymbol x) = \prod_{s=1}^d \rho_1(x_s)
> 0 
\quad 
(\boldsymbol x= (x_1, \dots, x_d)\in {\bf R}^d),
$$
by $\sqrt d\|\boldsymbol \xi\|\ge\|\boldsymbol \xi\|_2$,  we obtain
$$
\widehat 
\rho(\boldsymbol \xi) = 0
\quad
\hbox{for}
\quad 
\boldsymbol \xi \in \widehat {\mathbf R}^d
\quad
\hbox{with}
\quad
\|\boldsymbol \xi  \|\ge 1/(\pi\sqrt d\,)
\quad\hbox{or}\quad
\|\boldsymbol \xi  \|_2\ge1 /\pi
.$$

\begin{lemma}
Assume that $A$ satisfies (\ref{Eq:expx}), 
For any bounded measurable set $\Gamma\subset \mathbf R^d$, 
and for any trigonometric polynomial $f_a$  satisfying (\ref{Eq:periodmean}), 
there exists a constant $D_{5, \Gamma, A,f_a}$ such that 
\begin{equation}
\int_\Gamma \max_{n\in \Delta}
\biggl(\sum_{k\in \Delta: k\le n} f_a(A^k \boldsymbol x)\biggr)^4 \,d\boldsymbol x
\le D_{5, \Gamma, A, f_a}({}^\# \Delta)^2
\label{Eq:maxfourth}
\end{equation}
for any finite set $\Delta \subset \mathbf N$. 
\end{lemma}

\begin{proof}
Since we can take a constant $D_{6, \Gamma}< \infty$ such that 
${\boldsymbol 1}_{\Gamma}(\boldsymbol x) \le D_{6, \Gamma}\, \rho(\boldsymbol x)$, 
it is sufficient to prove 
\begin{equation*}
\int_{\mathbf R^d} \max_{n \in \Delta}
\biggl(\sum_{k\in \Delta: k\le n} f_a(A^k \boldsymbol x)\biggr)^4 \rho(\boldsymbol x)\,d\boldsymbol x
\le D_{7, A, f_a}C (^\#\Delta)^2.
\end{equation*}

Koml\'os-R\'ev\'esz \cite{komlosrevesz}
proved the following: 
Suppose that $X$ is a non-empty  set,  $m$ is a $\sigma$-finite measure on $X$, 
and $\{\varphi_i\}$ is a sequence of real valued measurable functions
satisfying 
\begin{equation}
\int_X \varphi_{k_1}^4 \,dm
\le M, 
\quad
\int_X\varphi_{k_1}\varphi_{k_2}\varphi_{k_3}\varphi_{k_4}\,dm=0
\quad 
(k_1 < k_2< k_3 < k_4)
,
\label{Eq:4MS}
\end{equation}
for some $M<\infty$, 
then there exists an absolute constant $D_8$ such that 
\begin{equation}
\int_X\biggl(\sum_{k=1}^N c_k\varphi_k \biggr)^4\,dm
\le D_8 \biggl(\sum_{k=1}^N c_k^2\biggr)^2
\quad(N\in {\bf N})
.
\label{Eq:fourth}
\end{equation}
Noting this estimate and by applying 
Erd\H os-Ste\v ckin Theorem (See \cite{moricz}), we can derive
$$
\int_X \max_{n\le N}\biggl(\sum_{k=1}^n c_k\varphi_k \biggr)^4\,dm
\le D_9 \biggl(\sum_{k=1}^N c_k^2\biggr)^2
\quad(N\in {\bf N})
,$$
where $D_9$ is an absolute constant. 
Note that any subsequence of $\{\varphi_k\}$ satisfying (\ref{Eq:4MS}) 
also satisfies (\ref{Eq:4MS})
and our version 
$$
\int_X \max_{n\in \Delta}\biggl(\sum_{k\in \Delta: k\le n} c_k\varphi_k \biggr)^4\,dm
\le D_9 \biggl(\sum_{k\in \Delta} c_k^2\biggr)^2
$$
follows. 

We use the  expression 
$$
f_a(\boldsymbol x) 
=
\sum_{\boldsymbol\xi\in R_a}
|\widehat f(\boldsymbol \xi)|
\cos (2\pi \boldsymbol \xi \boldsymbol x + \gamma_{\boldsymbol \xi})
$$
to have
$$
\sum_{k=1}^N f_a(A^k\boldsymbol x) 
=
\sum_{\boldsymbol\xi\in R_a}
|\widehat f(\boldsymbol \xi)|
\sum_{k=1}^N
\cos (2\pi \boldsymbol \xi A^k\boldsymbol x + \gamma_{\boldsymbol \xi})
.$$
Because of 
$\sum_{\boldsymbol\xi\in R_a}|\widehat f(\boldsymbol \xi)|< \infty$, 
if we have (\ref{Eq:fourth}) for 
$\varphi_k({}\cdot{}) = \cos (2\pi \boldsymbol \xi A^k {}\cdot{} + \gamma_{\boldsymbol \xi})$, 
then we have (\ref{Eq:fourth}) for 
$\varphi_k({}\cdot{}) = f_a(A^k {}\cdot{})$. 
Hence  it is enough to prove that 
the sequence 
$\{\cos (2\pi \boldsymbol \xi A^k\boldsymbol x + \gamma_{\boldsymbol \xi})\}_{k\in \mathbf N}$
satisfies (\ref{Eq:4MS}) under the measure $\rho(\boldsymbol x)\,d\boldsymbol x$.

Take $p\in \mathbf N$ large enough to satisfy
$$
\Lambda ^{3p} \ge (3 D_{3,a}^{-1} D_{1,Q} / \pi)\vee 3^3.
$$
For  $r=0$, $1$, \dots, $p-1$, 
we show that 
$\{\cos (2\pi \boldsymbol \xi A^{kp-r}\boldsymbol x + \gamma_{\boldsymbol \xi})\}_{k\in \mathbf N}$
satisfies (\ref{Eq:4MS}) under the measure $\rho(\boldsymbol x)\,d\boldsymbol x$. 
Then we can see that it satisfies (\ref{Eq:fourth}), 
and then by using Minkowski's inequality
we see 
that 
$\{\cos (2\pi \boldsymbol \xi A^{k}\boldsymbol x + \gamma_{\boldsymbol \xi})\}_{k\in \mathbf N}$
itself satisfies (\ref{Eq:fourth}). 

We first note that for $h\le \beta$ and $k_1< k_2< k_3 < k_4$, 
\begin{align*}
&|\lambda_h^{k_4p-r} \pm \lambda_h^{k_3p-r} \pm \lambda_h^{k_2p-r} \pm \lambda_h^{k_1p-r}|
\\&\ge 
|\lambda_h^{k_4p-r}| - |\lambda_h^{(k_4-1)p-r}| - |\lambda_h^{(k_4-2)p-r} |
- \cdots 
\\&
\ge
\Bigl(1 - \frac1{|\lambda_h^p|-1}\Bigr)|\lambda_h^{k_4p-r}|
\ge
\Bigl(1 - \frac1{\Lambda^p-1}\Bigr)\Lambda^{k_4p-r}
\\&\ge 
\frac12 \Lambda^{k_4p-r}
\ge \frac12 \Lambda^{3p}.
\end{align*}

By putting $\varsigma_4=1$, we have
\begin{align*}
&
\prod_{s =1}^4
\cos(2\pi \boldsymbol \xi A^{k_s  p-r} \boldsymbol x+\gamma_{\boldsymbol \xi})
\\
&=
8^{-1}
\sum_{\varsigma_1, \varsigma_2, \varsigma_3= \pm1}
\exp\biggl(2\pi \iu 
\sum_{s =1}^4 \varsigma_s (\boldsymbol\xi Q  B^{k_s  p-r}
Q^{-1}\boldsymbol x +\gamma_{\boldsymbol \xi})
\biggr)
\\
&
=
8^{-1}
\sum_{\varsigma_1, \varsigma_2, \varsigma_3= \pm1}
\exp(2\pi \iu (\varsigma_1 +\cdots+ \varsigma_4) \gamma_{\boldsymbol \xi})
\\
&\qquad\quad
\times
\exp\biggl(2\pi \iu \biggl(
\widetilde{\boldsymbol\xi}_1 \sum_{s =1}^4 \varsigma_s  B_1^{k_s  p-r}, 
\dots, 
\widetilde{\boldsymbol\xi}_\beta \sum_{s =1}^4 \varsigma_s  B_\beta^{k_s  p-r}
\biggr)
Q^{-1}\boldsymbol x\biggr).
\end{align*}

Suppose that $h(\boldsymbol \xi)\le \alpha$. 
By denoting $h(\boldsymbol \xi)$ simply by $h$, denoting 
$\widetilde{\boldsymbol\xi}_h$ by\break $(\xi_1, \dots, \xi_{d_h})$, and 
by taking a $\delta$ such that $\xi_1= \dots=\xi_{\delta-1}=0\not= \xi_{\delta}$, 
by $|\xi_{\delta}|=m_1(\widetilde {\boldsymbol \xi}_h)\ge D_{3, a}$, 
we have
\begin{align*}
&\Biggl|
\biggl(
\widetilde{\boldsymbol\xi}_h
\sum_{s =1}^4 \varsigma_s  B_h^{k_s  p-r}
\biggl)_{\delta}
\Biggr|
=
\biggl|\xi_\delta\sum_{s =1}^4 \varsigma_s  \lambda_h^{k_s  p-r} 
\biggr|
\ge  \frac{D_{3,a}\Lambda^{3p}}{2}
.
\end{align*}
Hence we have
\begin{align*}
&\biggl\|
\boldsymbol\xi
Q
\sum_{s =1}^4
\varsigma_s  B^{k_s  p-r}
Q^{-1}
\biggr\|
\ge  \frac{D_{1,Q}^{-1}D_{3,a}\Lambda^{3p}}{2}
\ge \frac{1}{\pi}
,
\end{align*}
which implies
\begin{equation}
\int_{\mathbf R^d}
\prod_{s =1}^4
\cos(2\pi \boldsymbol \xi A^{k_s  p-r} \boldsymbol x+\gamma_{\boldsymbol \xi})
\rho(\boldsymbol x)\,d\boldsymbol x
=0.
\label{Eq:cosMS}
\end{equation}

Suppose that $h(\boldsymbol \xi)> \alpha$. 
By denoting $h(\boldsymbol \xi)$ simply by $h$, denoting 
$\widetilde{\boldsymbol\xi}_h$ by\break 
$(\boldsymbol\eta_1, \dots, \boldsymbol\eta_{d_h/2})$, 
and by taking a $\delta$ such that $\boldsymbol\eta_1= \dots=\boldsymbol\eta_{\delta-1}=\boldsymbol 0\not= 
\boldsymbol\eta_{\delta}$, we have
\begin{align*}
&\biggl\|
\biggl(
\widetilde{\boldsymbol\xi}_h
\sum_{s  =1}^4 \varsigma_s  B_h^{k_s  p-r}
\biggr) _{2\delta-1, 2\delta}
\biggr\|_2
\\&=
\biggl\|\boldsymbol\eta_\delta\sum_{s =1}^4 
\varsigma_s  \lambda_h^{k_s  p-r}Z_{\theta (k_s  p-r)}
\biggr\| _2
\\&
\ge 
\|\boldsymbol\eta_\delta\|_2
\bigl(|\lambda_h|^{k_4p-r} - |\lambda_h|^{k_3p-r} 
-|\lambda_h|^{k_2p-r} -|\lambda_h|^{k_1p-r}\bigr)
\\&
\ge   \frac{m_2(\widetilde {\boldsymbol \xi}_h)\Lambda^{3p}}{2}
\ge \frac{D_{3,a}\Lambda^{3p}}{2}
.
\end{align*}
Hence in the same way as before, we can verify (\ref{Eq:cosMS}). 
\end{proof}

\section{Martingale approximation}

Take $L_\delta^{(h)}\in \mathbf R$ 
($\delta=1$, \dots, $d_h$, $h=1$, \dots $\beta$)
and $L>0$ arbitrarily 
and put
\begin{equation}\label{Eq:Omega}
\Omega=
\biggl\{
\sum_{h=1}^\beta
\sum_{\delta=1}^{d_h}
t^{(h)}_\delta \boldsymbol q^{(h)}_\delta
\biggm| 
L^{(h)}_{\delta} \le t^{(h)}_{\delta} < L^{(h)}_{\delta}+L
\biggr\}
.
\end{equation}
Let $\mathcal F$ be the Borel $\sigma$-field on $\Omega$ and 
put
$$
P_\Omega(B) = \frac{\Leb (B)}{ \Leb(\Omega)}
\quad (B\in \mathcal F)
.$$
We consider the sequence $\{f_a(A^k{}\cdot{})\}$  on 
the probability space $(\Omega, \mathcal F, P_\Omega)$. 
We state the almost sure invariance principle 
for the sequence. 
We denote the Lebesgue measure on $[\,0,1)$ by $\leb$. 

\begin{proposition}\label{prop:asip}
Let $A$ be a $d\times d$ real matrix satisfying (\ref{Eq:expx}), 
and let $f_a$ be a trigonometric polynomial  on $\mathbf R^d$ 
 satisfying (\ref{Eq:periodmean}). 
By taking the product probability space $(\Omega\times [\,0,1), \mathcal F\otimes 
\mathcal B([\,0,1)), P_\Omega\times\leb)$ and regard the sequence $\{f_a(A^k{}\cdot{})\}$ defined
on this space. If $\sigma^2(f_a)> 0$, 
then we can define a sequence $\{Z_i\}$ of standard normal i.i.d.  such that
\begin{equation}
\sum_{k=1}^N f_a(A^k {}\cdot{})
=
\sum_{i\le N\sigma^2(f_a)} Z_i 
+o(N^{62/125})
\quad
\hbox{a.s.}
\label{Eq:asip}
\end{equation}
\end{proposition}

From this proposition, we can derive the central limit theorem.

\begin{corollary}\label{cor:cltlil}
Let $A$ be a $d\times d$ real matrix satisfying (\ref{Eq:expx}), 
and let $f_a$ be a trigonometric polynomial  on $\mathbf R^d$ 
 satisfying (\ref{Eq:periodmean}). 
Then for any probability measure $P$ on $\mathbf R^d$ which is absolutely continuous 
with respect to Lebesgue measure, 
on the probability space $(\mathbf R^d, \mathcal B^d, P)$ we have the convergence in law
\begin{equation}
\frac1{\sqrt N} 
\sum_{k=1}^N 
f_a(A^k {}\cdot{})
\buildrel {\mathcal D}\over \longrightarrow
N(0, \sigma^2(f_a))
\quad(N\to\infty).
\label{Eq:clt}
\end{equation}
\end{corollary}

\begin{proof}[Proof of Proposition \ref{prop:asip}]

We divide the increasing sequence $\bf N$ of positive integers 
into consecutive blocks
$$
{\bf N}= \Delta_1' \cup \Delta_1 \cup \Delta_2'\cup \Delta_2\cup\dots
$$
where 
$$^\# \Delta_i = \lfloor i^{2/3}\rfloor
\quad\hbox{and}\quad 
{}^\# \Delta_i'= \lfloor 1 + (9+5d/3)\log _{\Lambda} i\rfloor.
$$
Put $i^- = \min \Delta_i$ and $i^+ = \max \Delta _i$. 
Clearly we have 
$$i^+ \le (2+(9+5d/3)/\log \Lambda )i^{5/3}
\quad\hbox{and}\quad
i^- -(i-1)^+ = {}^\#\Delta_i'
.$$

Put 
$$
\mu_h (i) = \lfloor \log _2 (i^{4+5d/3} |\lambda_h| ^{i^+})\rfloor.
$$
For  $i\in \mathbf N$, $1\le h\le \beta$, $1\le \delta\le d_h$, 
 and $j_\delta^{(h)} =0$, \dots, $2^{\mu_h{(i)}}-1$, 
we set
\begin{align*}
&J\bigl(i, (j^{(1)}_1, \dots, j^{(1)}_{d_1}), (j^{(2)}_1, \dots, j^{(2)}_{d_2}), \dots , 
(j^{(\beta)}_1, \dots, j^{(\beta)}_{d_\beta})\bigr)
\\
&\quad=
\biggl\{
\sum_{h=1}^\beta
\sum_{\delta=1}^{d_h}
(L^{(h)}_{\delta}+L 2^{-\mu_h(i)}(j^{(h)}_\delta+ t^{(h)}_\delta)) \boldsymbol q^{(h)}_\delta
\biggm| 
0 \le t^{(h)}_{\delta} < 1 
\biggr\}
\end{align*}
and denote the collection of all such cubes by $\mathcal J(i)$. 
Let $\mathcal F_i $ be the $\sigma$-field on 
$\Omega$ generated by $\mathcal J(i)$. 
$\{\mathcal F_i\}$ forms a filtration on 
$(\Omega, \mathcal F, P_\Omega)$.
Let 
$$\widetilde {\mathcal F}_i= \{F\times [\,0,1)\mid F\in \mathcal F_i\}.$$
Clearly $\{\widetilde {\mathcal F}_i\}$ forms a filtration on 
$(\Omega\times [\,0,1), \mathcal F\times\mathcal B[\,0,1), P_\Omega\times \leb)$.
 
For $(\boldsymbol x, x)\in \Omega\times [\,0,1)$, we here put
$$
\widetilde T_i(\boldsymbol x, x) 
=
T_i(\boldsymbol x) 
=
\sum_{k\in \Delta_i}
f_a(A^k \boldsymbol x)
$$
and 
prove
\begin{equation}
\widetilde E(\widetilde T_i\mid \widetilde {\mathcal F}_{i-1})(\boldsymbol x, x) = 
E(T_i\mid \mathcal F_{i-1})(\boldsymbol x) = O(i^{-4}), 
\label{Eq:ticond}
\end{equation}
where $\widetilde E({}\cdot{}\mid{}\cdot{})$ denotes the conditional expectation on 
$\Omega\times[\,0,1)$ and $E({}\cdot{}\mid{}\cdot{})$ that on $\Omega$. 
The first equality is trivial. 
Take  $\boldsymbol x\in \Omega$ arbitrarily and take $J\in \mathcal J(i-1)$
such that $\boldsymbol x\in J$. 
We note that 
$$
E(X\mid \mathcal F_{i-1})(\boldsymbol x) =  \frac1{\Leb (J)}\int_J X(\boldsymbol y)\,d\boldsymbol y
.$$
By putting
$$
R_i=
\begin{pmatrix}
L2^{-\mu_1(i)} E_{d_1}& O & \cdots & O
\\
O   &L2^{-\mu_2(i)}E_{d_2}& \ddots & \vdots
\\
\vdots & \ddots & \ddots & O
\\
O   & \cdots & O & L2^{-\mu_\beta(i)}E_{d_\beta}
\end{pmatrix}, 
$$
where $E_\gamma$ is the unit matrix of size $\gamma\times \gamma$, 
we can write
$$J= \{\boldsymbol b + Q R_{i-1} \boldsymbol t \mid \boldsymbol t\in [\,0,1)^d\}$$
by using some $\boldsymbol b\in {\bf R}^d$. 
Changing variables by $\boldsymbol y = \boldsymbol b + QR_{i-1} \boldsymbol t$ 
and noting 
$$\frac{\partial \boldsymbol y }{\partial \boldsymbol t}=|\det(QR_{i-1})|=\Leb (J),$$
we have
\begin{align*}
&E(\exp(2\pi \iu \boldsymbol \xi A^k {}\cdot{})\mid \mathcal F_{i-1})(\boldsymbol x)
\\&
=
\frac{1}{\Leb(J)}
\int_J \exp(2\pi \iu \boldsymbol \xi A^k \boldsymbol y) \,d\boldsymbol y
\\
&=
\int_{[0,1)^d}
\exp(2\pi \iu \boldsymbol \xi Q B^k Q^{-1}(\boldsymbol b + QR_{i-1} \boldsymbol t)) \,d\boldsymbol t
\\&
=
{\exp(2\pi \iu c)}
\prod_{h=1}^\beta
\int_{[0,1)^{d_h}}
\exp(2\pi \iu  L2^{-\mu_h(i-1)} \widetilde{\boldsymbol \xi}_h B_h^k  \boldsymbol t_h) \,d\boldsymbol t_h
,\end{align*}
where 
$c=\boldsymbol \xi Q B^k Q^{-1}\boldsymbol b $, and 
$\widetilde{\boldsymbol\xi}_h\in \widehat {\bf R}^{d_h}$ and $\boldsymbol t_h\in {\bf R}^{d_h}$ 
are given by
$
\boldsymbol \xi Q
=
(\widetilde{\boldsymbol\xi}_1 , \dots, \widetilde{\boldsymbol\xi}_\beta)
$
and 
$
\boldsymbol t
=
\begin{pmatrix}
\boldsymbol t_1
\\
 \vdots
\\
\boldsymbol t_\beta
\end{pmatrix}
$.
If we write $\widetilde{\boldsymbol \xi}_h B_h^k = (\zeta^{(h)}_1, \dots, \zeta^{(h)}_{d_h})$, 
we have
\begin{align*}
&\int_{[0, 1)^{d_h}}
\exp(2\pi \iu  L 2^{-\mu_h(i-1)} \widetilde{\boldsymbol \xi}_h  B_h^k  \boldsymbol t_h) \,d\boldsymbol t_h
\\&
=
\prod _{\delta=1}^{d_h}
\int_{0}^{1}
\exp(2\pi \iu L2^{-\mu_h(i-1)}\zeta^{(h)}_\delta  t)\,dt
\\&=
\prod _{\delta=1}^{d_h}
\phi(\pi L2^{-\mu_h(i-1)}\zeta^{(h)}_\delta )
\exp(\pi \iu c' ),
\end{align*}
where $c'=L 2^{-\mu_h(i-1)}\zeta^{(h)}_\delta$, 
$\phi(x)=(\sin x)/x$ if $x\not=0$ and $\phi(0) = 1$. 

By (\ref{Eq:lowerb}), 
 there exists a $\delta(\boldsymbol \xi)$ such that 
$|\zeta^{(h(\boldsymbol \xi))}_{\delta(\boldsymbol \xi)}|
\ge 
{D_{3,a}}|\lambda_{h(\boldsymbol \xi)}|^k/2 
$. 
Hence we have
\begin{align*}
\phi(\pi L 2^{-\mu_{h(\boldsymbol \xi)}(i-1)}\zeta^{(h(\boldsymbol \xi))}_{\delta(h(\boldsymbol \xi))})
&\le 
2/\pi L 2^{-\mu_{h(\boldsymbol \xi)(i-1)}}D_{3,a}|\lambda_{h(\boldsymbol \xi)}|^k   
\\&
\le 
{2(i-1)^{4+5d/3}|\lambda_{h(\boldsymbol \xi)}|^{(i-1)^+}
}/{
 \pi D_{3,a}|\lambda_{h(\boldsymbol \xi)}|^{i^-}L}
\\&
\le 
{2i^{4+5d/3}\Lambda^{(i-1)^+-i^-}
}/{ \pi D_{3,a} L}
=O(i^{-5}).
\end{align*}
By 
$$
T_i(\boldsymbol y) = \sum_{k\in \Delta_i}\sum_{\boldsymbol \xi\in R_a}
\widehat f(\boldsymbol \xi) \exp(2\pi \iu \xi A^k \boldsymbol y)
,$$
we have
(\ref{Eq:ticond}). 

Secondly, we prove
\begin{equation}
\widetilde E(\widetilde T_i \mid \widetilde {\mathcal F}_i) (\boldsymbol x,x) 
- \widetilde T_i(\boldsymbol x,x) =
E(T_i \mid \mathcal F_i) (\boldsymbol x) - T_i(\boldsymbol x) = O(i^{-3})
.
\label{Eq:tiapprox}
\end{equation}
Assume that  $k\in \Delta_i$ and $\boldsymbol x\in J\in \mathcal J(i)$.
Again the first equality is trivial. We have
\begin{align*}
E(f_a(A^k {}\cdot {})\mid \mathcal F_i) (\boldsymbol x) - f_a(A^k \boldsymbol x)
&=
\frac1{\Leb(J)} \int _J (f_a(A^k \boldsymbol y) - f_a(A^k \boldsymbol x))\,d\boldsymbol y.
\end{align*}
By $\boldsymbol x$, $\boldsymbol y\in J$, we have
$$
\boldsymbol y - \boldsymbol x = 
\sum_{h=1}^\beta L2^{-\mu_h(i)} \sum_{\delta=1}^{d_h} t^{(h)}_\delta  \boldsymbol q^{(h)}_\delta  
$$ 
for some $-1< t^{(h)}_\delta<1$. 
By (\ref{Eq:upperb}), we have
$$\|A^k \boldsymbol q^{(h)} _\delta\|
\le D_{2,A} (\max_{\delta, h} \|\boldsymbol q^{(h)} _\delta\|) 
|\lambda _h|^{i^+} (i^+)^d
$$ 
and 
\begin{align*}
\|A^k \boldsymbol y - A^k \boldsymbol x\|
&\le 
\sum_{h=1}^\beta L 2^{-\mu_h(i)} \sum_{\delta=1}^{d_h} \|A^k \boldsymbol q^{(h)} _\delta\|
=O(i^{-4}). 
\end{align*}
 Lipschitz continuity of $f$ implies
$$
\bigl|E(f_a(A^k {}\cdot {})\mid \mathcal F_i) (\boldsymbol x) - f_a(A^k \boldsymbol x)\bigr|
=O(i^{-4})
,$$
and thereby (\ref{Eq:tiapprox}). 

Put 
$$Y_i = E(T_i \mid \mathcal F_i)- E(T_i \mid \mathcal F_{i-1})
\quad\hbox{and}\quad
\widetilde Y_i = \widetilde E(\widetilde T_i \mid \widetilde{\mathcal  F}_i)
- \widetilde E(\widetilde T_i \mid  \widetilde{\mathcal F}_{i-1}).
$$
Clearly $\{Y_i, \mathcal F_i\}$ and $\{\widetilde Y_i, \widetilde{\mathcal F}_i\}$ 
are martingale differences and $\widetilde Y_i(\boldsymbol x, x)= Y_i(\boldsymbol x)$. 
By combining (\ref{Eq:ticond}) and (\ref{Eq:tiapprox}), 
we have
\begin{equation}
\|\widetilde Y_i - \widetilde T_i\|_\infty =
\|Y_i - T_i\|_\infty =O(i^{-3}).
\label{Eq:ytapprox}
\end{equation}

By $\|T_i\|_\infty = O(i)$, we have $\|E(T_i \mid \mathcal F_i)\|_\infty$, 
$\|E(T_i \mid \mathcal F_{i-1})\|_\infty=
O(i)$, and $\|Y_i\|_\infty= O(i)$, which implies $\|Y_i+T_i\|_\infty=O(i)$ 
and 
$$\|Y_i^2-T_i^2\|_\infty=O(i^{-2}).$$
By $\|Y_i^2+T_i^2\|_\infty=O(i^2)$, we have
$$\|Y_i^4-T_i^4\|_\infty=O(1).$$
By the last inequality and (\ref{Eq:maxfourth}), 
we have
\begin{equation}\label{Eq:yifourth}
\widetilde E\widetilde Y_i^4 
=
\widetilde E\widetilde T_i^4 +O(1)
=
 E T_i^4 +O(1)
=
O(i^{4/3}).
\end{equation}

We have
\begin{align*}
&\biggl(\sum_{k\in \Delta_i} f_a(A^k \boldsymbol x) \biggr)^2 
=
\sum_{\boldsymbol \xi\in R_a}
\sum_{\boldsymbol \xi'\in R_a}
\widehat f(\boldsymbol \xi)
\widehat f(\boldsymbol \xi')
\sum_{k\in \Delta_i}
\sum_{k'\in \Delta_i}
\exp(2\pi \iu (\boldsymbol \xi A^k + \boldsymbol \xi' A^{k'})\boldsymbol x)
.
\end{align*}
Put
$$
v_i = 
\sum_{\boldsymbol \xi\in R_a}
\sum_{\boldsymbol \xi'\in R_a}
\widehat f(\boldsymbol \xi)
\widehat f(\boldsymbol \xi')
\sum_{k\in \Delta_i}
\sum_{k'\in \Delta_i}
{\mathbf 1}(\boldsymbol \xi A^k + \boldsymbol \xi' A^{k'}=\boldsymbol 0)
.$$
There exists an $l_0$ such that for all $\boldsymbol \xi \in R_a$ and $\boldsymbol \xi' \in R_a$, 
  $\boldsymbol \xi A^l + \boldsymbol \xi' \not =\boldsymbol 0$ 
and 
  $\boldsymbol \xi  + \boldsymbol \xi' A^l\not=\boldsymbol 0$ 
hold for $l> l_0$. 
If $l\le l_0$ and 
$\boldsymbol \xi  + \boldsymbol \xi'A^l =\boldsymbol 0$, we have
$$
0\le i-{} ^\#\{(k, k')\in \Delta_i^2\mid 
\boldsymbol \xi A^k + \boldsymbol \xi' A^{k'}=\boldsymbol 0\} 
=i- (i-l+1)\vee 0\le l_0. 
$$
By noting 
$$
i\sigma^2(f_a)
=
i
\sum_{\boldsymbol \xi\in R_a}
\sum_{\boldsymbol \xi'\in R_a}
\widehat f(\boldsymbol \xi)
\widehat f(\boldsymbol \xi')
{\mathbf 1}\biggl(
{\boldsymbol \xi A^l + \boldsymbol \xi' =\boldsymbol 0
\hbox{ or }
\boldsymbol \xi  + \boldsymbol \xi' A^l=\boldsymbol 0
\atop
\hbox{ for some }l\ge 0\hfill}
\biggr), 
$$
we have
\begin{equation}
\label{Eq:viap}
| v_i - i\sigma^2(f_a)|
\le
l_0
\sum_{\boldsymbol \xi\in R_a}
\sum_{\boldsymbol \xi'\in R_a}
|\widehat f(\boldsymbol \xi)
\widehat f(\boldsymbol \xi')|. 
\end{equation}

For $\boldsymbol x\in \Omega$, take $J\in \mathcal F_{i-1}$ such that $\boldsymbol x\in J$. 
Under the condition $\boldsymbol \xi A^k + \boldsymbol \xi' A^{k'}\not=\boldsymbol 0$, 
in the same way as before, we have
\begin{align*}
&E( \exp(2\pi \iu (\boldsymbol \xi A^k + \boldsymbol \xi' A^{k'}){}\cdot{})
\mid
\mathcal F_{i-1})(\boldsymbol x)
\\&
=
\frac{1}{\Leb(J)}
\int_J \exp(2\pi \iu (\boldsymbol \xi A^k + \boldsymbol \xi' A^{k'})\boldsymbol y)\,d\boldsymbol y
\\&=
{\exp(2\pi \iu c) }
\prod_{h=1}^\beta
\int_{[0,1)^{d_h}}
\exp(2\pi \iu   (\widetilde{\boldsymbol \xi}_hB_h^{k-k'}+ \widetilde{\boldsymbol \xi}'_h)
 B_h^{k'} L2^{-\mu_h(i-1)} \boldsymbol t_h) \,d\boldsymbol t_h
\end{align*}
and  by 
$m_{i(h({\boldsymbol \xi}))}(\widetilde {\boldsymbol \xi}_{h({\boldsymbol \xi})} 
B_{h({\boldsymbol \xi})}^{k-k'} + \widetilde {\boldsymbol \xi}'_{h({\boldsymbol \xi})})
\ge D_{4,a}$, as before we have
$$
E( \exp(2\pi \iu (\boldsymbol \xi A^k + \boldsymbol \xi' A^{k'}){}\cdot{})
\mid
\mathcal F_{i-1})(\boldsymbol x)
= O(i^{-5})
.$$
Since the number of choices of $(\boldsymbol \xi, \boldsymbol \xi')$ is finite 
and the number of choices of $(k,k')$ is at most $i$, 
we have
$$
E( T_i^2 - v_i
\mid
\mathcal F_{i-1})(\boldsymbol x)
= O(i^{-4})
.$$
By putting 
$$\beta_M = \sum_{i=1}^M v_i
\quad\hbox{and}\quad
l_M = {}^\#\Delta_1+ \cdots +{}^\#\Delta_M
,$$
we have
$$
\biggl\|
\sum_{i=1}^M
E(T^2_i \mid \mathcal F_{i-1}) -\beta_M
\biggr\|_\infty
=O(1)
\quad\hbox{and}\quad
|\beta_M - l_M \sigma^2(f_a)|
\le D_{10,f_a}M
$$
for some $D_{10,f_a}< \infty$. 
Since we have
$$
\biggl\|
\sum_{i=1}^M
\bigl(
E(T^2_i \mid \mathcal F_{i-1}) 
-
E(Y^2_i \mid \mathcal F_{i-1}) 
\bigr)
\biggr\|_\infty
\le
\sum_{i=1}^M
\|T^2_i- Y^2_i\|_\infty
=O(1)
,$$
we have
\begin{equation}
\|\widetilde V_M -\beta_M\|_\infty \le D_{11},
\label{Eq:vmasymp}
\end{equation}
where 
$$\widetilde V_M = \sum_{i=1}^M \widetilde E(\widetilde Y^2_i \mid \widetilde {\mathcal F}_{i-1}) 
=
\sum_{i=1}^M E(Y^2_i \mid \mathcal F_{i-1}) 
\quad\hbox{and}\quad D_{11}<\infty.$$

Now we use the theorem which is a version of Strassen's theorem (Theorem 4.4 of \cite{1967S}).

\begin{theorem}[Monrad-Philipp \cite{1991MPh} Theorem 7. 
This version is Lemma A.4 in  Philipp \cite{baker}]
Let  $\{  \widetilde Y_i,  \widetilde {\mathcal F}_i\}$ be  
a square integrable martingale difference 
satisfying
$$
  \widetilde V_M = \sum_{i=1}^M \widetilde E(  \widetilde Y_i^2 \mid   \widetilde {\mathcal F}_{i-1}) 
\to \infty 
\ \hbox{a.s.}\ \hbox{ and }\
\sum_{i=1}^\infty \widetilde E\biggl( 
\frac{  \widetilde Y_i^2 
{\bf 1}  _{\{   \widetilde Y_i^2 \ge \psi(  \widetilde V_i) \}}
} { \psi(\widetilde   V_i)} \biggr) 
< \infty
$$
for some non-decreasing  $\psi$ 
such that 
$\psi(x)(\log x)^{\alpha}/ x$ is non-increasing for some  $\alpha > 50$
and   $\lim_{x\to \infty}\psi(x)=\infty$. 
If there exists a uniformly distributed random variable
$\widetilde U$ which is independent of  $\{\widetilde   Y_n\}$, 
there exists a sequence $\{Z_i\}$ of standard normal i.i.d. 
such that 
\begin{equation}
\sum_{i\ge 1}
 \widetilde  Y_{i}{\bf 1}_{\{  \widetilde V_i\le t\}}
= 
G_t
+ o\bigl( t^{1/2} (\psi(t)/ t) ^{1/50}\bigr)
\quad(t\to\infty)\quad\hbox{a.s.,}
\label{Eq:monrad}
\end{equation}
where 
$$G_t = \sum_{i\le t}  Z_i.$$
\label{thm:monrad}
\end{theorem}

From now on, we regard $f_a(A^k \boldsymbol x)$ as a random variable on 
$\Omega\times [\,0,1)$. 

Recall that $ \sigma^2(f_a)> 0$ and put $\psi(x) = x^{4/5}$. 
One can see that 
$$
\beta_M \sim l_M \sigma^2(f_a)
\quad\hbox{and}\quad
\widetilde V_M\to \infty
\quad\hbox{a.s.}
$$ by (\ref{Eq:viap}) and (\ref{Eq:vmasymp}). 

Since 
  $$ \widetilde  V_M\ge \beta_M-D_{11}\ge \sigma^2(f_a) l_M/2$$ holds for large $M$, 
we see by (\ref{Eq:yifourth}) 
that
$$
\sum\widetilde  E\bigl( 
{  \widetilde Y_i^2 {\bf 1}_{ \{   \widetilde Y_i^2 
\ge \psi(  \widetilde V_i) \}} }/{ \psi(  \widetilde V_i)} 
\bigr) 
\le
\sum
{\widetilde E  \widetilde Y_i^4 }/{ \psi^2(\sigma^2(f_a) l_i/2)}
\ll\sum { i^{4/3}}/{ l_i^2}< \infty.
$$
Because of $\beta_{M+1}-\beta_M = v_{M+1}\to\infty$, 
we obtain 
$$\widetilde V_M \le \beta_M + D_{11} < \beta_{M+1} - D_{11} < \widetilde V_{M+1}$$ 
for large M, and $\widetilde V_i \le \beta_M + D_{11}$ becomes equivalent to $i\le M$. 
By putting $t = \beta_M + D_{11}$ in (\ref{Eq:monrad}) and by noting (\ref{Eq:ytapprox})
we have
\begin{equation}
\sum_{i=1}^M   \widetilde T_i 
=
\sum_{i=1}^M   \widetilde Y_i +O(1)
=
G_{\widetilde V_M+D_{11}}
+ o\bigl(  l_M ^{249/250} \bigr), 
\quad\hbox{a.s.}
\label{Eq:yiasip}
\end{equation}

Put 
$$\Delta_M^\flat = \Delta_1\cup \dots \cup \Delta_M
\quad\hbox{and}\quad
\Delta_M^\natural= \Delta_1'\cup \dots \cup \Delta_M'
.$$
By $^\#\Delta_M^\natural= O(M\log M)$, 
we obtain 
$$i_M^+ = l_M + {}^\#\Delta_M^\natural \sim l_M.$$

Note that 
\begin{equation}
F_M^\flat :=
\max_{m\in \Delta_M} \biggl |\sum_{k=m}^{i_M^+} f_a(A^k {}\cdot{}) 
\biggr|
\le M^{2/3} \|f\|_\infty
=O(l_M^{2/5}).
\label{Eq:maxflat}
\end{equation}
We can prove
\begin{equation}
F_M^\natural
:=
\max _{m\in \Delta_M^\natural} \biggl|\sum_{k\in \Delta_M^\natural, k\le m} f_a(A^k {}\cdot{}) \biggr|
=
o(l_M^{19/40})
\quad\hbox{a.s.,}
\label{Eq:maxnatural}
\end{equation}
since 
 (\ref{Eq:maxfourth}) implies
$$
E\bigl(
(l_M^{-19/40} F_M^\natural)^4
\bigr)
= O(M^{-7/6} (\log M)^2)
,$$
and is summable in $M$. 

Hence for $N\in \Delta_M'\cup \Delta_M$, 
we obtain 
\begin{equation}
\biggl|
\sum_{k=1}^N f_a(A^k {}\cdot{})
-
\sum_{i=1}^M \widetilde T_i({}\cdot{})
\biggr|
\le F_M^\flat + F_M^\natural 
=
o(l_M^{19/40}) 
=
o(N^{19/40}).
\label{Eq:tierror}
\end{equation}
Now we apply  the following 
result on the fluctuation  of the standard Wiener process $W(t)$
due to Cs\"org\H o-R\'ev\'esz 
((1.2.4) in Theorem 1.2.1 of \cite{1981CsR}). 
For non-decreasing $a_T$ such that $0< a_T \le T$ and $T/a_T$ is 
non-decreasing, 
we have
\begin{equation*}
\varlimsup_{T\to \infty}
\sup_{0\le t\le T-a_T}
\sup_{0 \le s \le a_T}
\frac
{|W(t+s) - W(t)| }
{\sqrt{2a_T (\log (T/a_T) + \log \log T)}}
= 1
, \quad\hbox{a.s.}
\end{equation*}
By putting 
$$T= i_M^+ \sigma^2(f_a)\quad\hbox{and}\quad
a_T = {}^\#(\Delta_M'\cup \Delta_M)+ D_{10,f_a}M = O(M\log M),
$$
we have
\begin{equation}
\bigl|
G_{\sigma^2(f_a) N } 
- 
G_{\beta_M+ D_{11}}
\bigr|
=O(M^{1/2}\log M) 
= O(N^{3/10}\log N)
\quad\hbox{a.s.}
\label{Eq:zierror}
\end{equation}
By combining  (\ref{Eq:yiasip}), (\ref{Eq:tierror}) and (\ref{Eq:zierror}) 
we have (\ref{Eq:asip}). 
\end{proof}

\section{Variance control}
We first prove that the series in (\ref{Eq:lv}) is absolutely convergent. 
By using  the convention $\widehat f(\boldsymbol \xi ) = 0$ for $\boldsymbol \xi \notin {\bf Z}^d$, 
we have
$$|\widehat f(\boldsymbol \xi) \widehat f(\boldsymbol \xi')|
{\mathbf 1}(\boldsymbol \xi+ \boldsymbol \xi'A^l=\boldsymbol 0)
=
|\widehat f(\boldsymbol \xi'A^l) \widehat f(\boldsymbol \xi')|.$$
Hence by noting 
$$\|\boldsymbol \xi A^l\|\ge \|\boldsymbol \xi A^l\|_2/\sqrt d\ge q^l 
\|\boldsymbol \xi\|_2/\sqrt d
\ge q^l/\sqrt d$$
 for $\boldsymbol \xi\in {\bf Z}^d\setminus \{\boldsymbol 0\}$, we have
\begin{equation}\label{Eq:serconv}
\begin{aligned}
&\sum_{l\ge 0}  \sum_{\boldsymbol \xi, \boldsymbol \xi'\in {\bf Z}^d}
|\widehat f(\boldsymbol \xi) \widehat f(\boldsymbol \xi')|
{\mathbf 1}(\boldsymbol \xi+ \boldsymbol \xi'A^l=\boldsymbol 0)
\\&=
\sum_{l\ge 0}  \sum_{\boldsymbol \xi\in {\bf Z}^d}
\bigl|\widehat f(\boldsymbol \xi) \widehat f(\boldsymbol \xi A^l)\bigr|
\\&\le
\sum_{l\ge 0}  
\biggl(\sum_{\boldsymbol \xi\in {\bf Z}^d}
\bigl|\widehat f(\boldsymbol \xi)\bigr|^2
\sum_{\boldsymbol \xi\in {\bf Z}^d}
\bigl| \widehat f(\boldsymbol \xi A^l)\bigr|^2
\biggr)^{1/2}
\\&\le 
\|f\|_{L^2[0,1)^d}
\sum_{l\ge 0}  
\|f-f_{q^l/\sqrt d}\|_{L^2[0,1)^d}
< \infty.
\end{aligned}
\end{equation}

Let $\Gamma \subset {\bf R}^d$ be a bounded measurable set.
Since we have $L^2$ convergence 
$f_a \to f$ 
on $[\,0,1)^d$ and by periodicity, we have the convergence 
on any bounded set $\Gamma'$. 

By changing variable 
we have $L^2(\Gamma)$ convergence 
$f_a(A^k \boldsymbol x) \to f(A^k \boldsymbol x)$. 
Hence we have $L^1(\Gamma)$ convergence 
 $f_a(A^k \boldsymbol x)f_a(A^{k'} \boldsymbol x) 
\allowbreak\to f(A^k \boldsymbol x)f(A^{k'} \boldsymbol x)$, 
and hence convergence in measure. 
Thus we have the convergence 
$$
\biggl(\sum_{k=1}^N f_a(A^k \boldsymbol x)\biggr)^2 
\to 
\biggl(\sum_{k=1}^N f(A^k \boldsymbol x)\biggr)^2$$
in measure on $\Gamma$ under the measure 
$\rho(\boldsymbol x)\,d\boldsymbol x$. 
That is why we can apply Fatou's Lemma and have 
\begin{align*}
&
\hskip-1em
\int_{\Gamma} 
\biggl( \sum_{k=1}^N f(A^k \boldsymbol x)\biggr)^2 \,d\boldsymbol x
\\&\le 
D_{6, \Gamma}
\int_{\Gamma} 
\biggl( \sum_{k=1}^N f(A^k \boldsymbol x)\biggr)^2 \rho(\boldsymbol x)\,d\boldsymbol x
\\&
\le
D_{6, \Gamma}
\varliminf_{a\to\infty}
\int_{\Gamma} 
\biggl( \sum_{k=1}^N f_a(A^k \boldsymbol x)\biggr)^2 \rho(\boldsymbol x)\,d\boldsymbol x
\\&
\le
2D_{6, \Gamma}
\sum_{\boldsymbol \xi, \boldsymbol \xi' \in {\bf Z}^d} 
\sum_{k\le l\le N}
\bigl|\widehat f(\boldsymbol \xi)\widehat f(\boldsymbol \xi')\bigr|
\mathbf 1 \bigl(\|\boldsymbol \xi A^k + \boldsymbol \xi' A^l \| \le 1/\pi\sqrt d\,\bigr).
\end{align*}
By  (\ref{Eq:expxi}), we obtain
$$\bigl\|\boldsymbol \xi A^k  + \boldsymbol \xi' A^l\bigr\|
\ge 
\bigl\|(\boldsymbol \xi   + \boldsymbol \xi' A^{l-k})A^k\bigr\|_2/\sqrt d
\ge
\bigl\|\boldsymbol \xi   + \boldsymbol \xi' A^{l-k}\bigr\|_2/\sqrt d
.$$
For $\boldsymbol \eta\in \mathbf R^d$, 
there is at most one $\boldsymbol \xi\in {\bf Z}^d$ such that 
$\|\boldsymbol \eta -\boldsymbol \xi\|_2\le 1/\pi$. 
In case such $\boldsymbol \xi$ exists, let $\chi(\boldsymbol \eta)= \boldsymbol \xi$, 
and $\chi(\boldsymbol \eta)= \boldsymbol 0$ otherwise. 
If $\boldsymbol \xi\not=\boldsymbol \xi'$, then 
$\|\boldsymbol \xi A^m- \boldsymbol \xi'A^m\|_2 \ge \|\boldsymbol \xi - \boldsymbol \xi'\|_2 
\ge 1$, and $\chi(\boldsymbol \xi A^m) \not=\chi( \boldsymbol \xi'A^m)$ if 
$\chi(\boldsymbol \xi A^m) \not=\boldsymbol 0$ and 
$\chi(\boldsymbol \xi'A^m) \not=\boldsymbol 0$. 
If $\chi(\boldsymbol \xi A^m) \not=\boldsymbol 0$, then 
$$\|\chi(\boldsymbol \xi A^m)\|\ge \|\chi(\boldsymbol \xi A^m)\|_2/\sqrt d
\ge (\|\boldsymbol \xi A^m\|_2- 1/\pi)/\sqrt d
\ge q^m/2\sqrt d
.$$
By 
\begin{align*}
\bigl|\widehat f(\boldsymbol \xi)\widehat f(\boldsymbol \xi')\bigr|
\boldsymbol 1\bigl(\bigl\|\boldsymbol \xi A^k  + \boldsymbol \xi' A^l\bigr\|\le 1/\pi\sqrt d\bigr)
&\le 
\bigl|\widehat f(\boldsymbol \xi)\widehat f(\boldsymbol \xi')\bigr|
\boldsymbol 1\bigl(\bigl\|\boldsymbol \xi A^k  + \boldsymbol \xi' A^l\bigr\|_2\le 1/\pi\bigr)
\\&=\bigl|\widehat f(\chi(\boldsymbol \xi' A^{k-l}))\widehat f(\boldsymbol \xi')\bigr|
,
\end{align*}
we have
\begin{align*}
&\sum_{\boldsymbol \xi, \boldsymbol \xi' \in {\bf Z}^d} 
\sum_{k\le l\le N}
\bigl|\widehat f(\boldsymbol \xi)\widehat f(\boldsymbol \xi')\bigr|
\boldsymbol 1\bigl(\bigl\|\boldsymbol \xi A^k  + \boldsymbol \xi' A^l\bigr\|\le 1/\pi\sqrt d\bigr)
\\&
\le 
 \sum_{k\le l\le N} 
\sum_{\boldsymbol \xi \in {\bf Z}^d} 
\bigl|\widehat f(\boldsymbol \xi)\widehat f(\chi(\boldsymbol \xi A^{k-l}))\bigr|
\le  
N \sum_{m\le N} 
\sum_{\boldsymbol \xi \in {\bf Z}^d} 
\bigl|\widehat f(\boldsymbol \xi)\widehat f(\chi(\boldsymbol \xi A^{m}))\bigr|
\\&
\le
N \sum_{m\le N} 
\biggl(\sum_{\boldsymbol \xi \in {\bf Z}^d} 
|\widehat f(\boldsymbol \xi)|^2
\sum_{\boldsymbol \xi \in {\bf Z}^d} 
|\widehat f(\chi(\boldsymbol \xi A^{m}))|^2\biggr)^{1/2}
\\&
\le N \|f\|_{L^2[0,1)^d}\sum_{m\le N} \|f-f_{q^m/2\sqrt d}\|_{L^2[0,1)^d} .
\end{align*}
We have proved
\begin{equation}\label{Eq:varcntl}
\int_{\Gamma} 
\biggl( \sum_{k=1}^N f(A^k \boldsymbol x)\biggr)^2 \,d\boldsymbol x
 \le \widehat CN \|f\|_{L^2[0,1)^d}
,\end{equation}
where $\widehat C= D_{6, \Gamma}\sum_{m=0}^\infty \|f-f_{q^m/2\sqrt d}\|_{L^2[0,1)^d}$.

If $\boldsymbol \xi A^{k'} + \boldsymbol \xi'\not=0$, by Riemann-Lebesgue Lemma we have
\begin{align*}
\int_\Gamma \exp(2\pi \iu (\boldsymbol \xi A^{k'} + \boldsymbol \xi') 
A^k \boldsymbol x)\,d\boldsymbol x
=
\widehat {\mathbf 1}_\Gamma
 ((\boldsymbol \xi A^{k'} + \boldsymbol \xi') A^k)
\to 0 
\end{align*}
as $k\to\infty$. 
Hence for a trigonometric polynomial $f_a$, we have (\ref{Eq:limvar}) as below: 
\begin{align*}
&\frac1{N\Leb(\Gamma)} \int_\Gamma
\biggl( \sum_{k=1}^N f_a(A^k  \boldsymbol x)\biggr)^2 \,d\boldsymbol x
\\&
=
\sum_{\boldsymbol \xi\in R_a}
\sum_{\boldsymbol \xi'\in R_a}
\widehat f(\boldsymbol \xi)
\widehat f(\boldsymbol \xi')
\sum_{l=0}^{N-1} 
\frac{2-\delta_{l,0}}{N\Leb(\Gamma)}
\sum_{k=1}^{N-l} 
\int_\Gamma
\exp(2\pi\iu (\boldsymbol \xi + \boldsymbol \xi' A^l) A^k \boldsymbol x)\,d\boldsymbol x
\\
&\to
\sum_{\boldsymbol \xi\in R_a}
\sum_{\boldsymbol \xi'\in R_a}
\widehat f(\boldsymbol \xi)
\widehat f(\boldsymbol \xi')
\sum_{l=0}^{\infty}
(2-\delta_{l,0}) 
{\mathbf 1}(\boldsymbol \xi + \boldsymbol \xi' A^l =\boldsymbol0)
=\sigma^2(f_a).
\end{align*}
Because of the absolute convergence (\ref{Eq:serconv}), 
we obtain
$$
\sigma(f_a) \to \sigma(f).
$$

\begin{proof}[Proof of Corollary \ref{cor:cltlil}]
Put $$
X_N^a(\boldsymbol x)= 
\frac1{\sqrt N}
\sum_{k=1}^N f_a(A^k \boldsymbol x)
.$$
First we assume $\sigma^2(f_a)> 0$ and 
prove the central limit theorem under the measure $P_\Omega$. 
By using Proposition \ref{prop:asip} we have
$$
X_N^a
-
{G_{N\sigma^2(f_a)}}/{\sqrt N}
\to 0
\quad\hbox{a.s.}
$$
and the law of  ${G_{N\sigma^2(f_a)}}/{\sqrt N}$ is $N(0, \sigma^2(f_a))$, 
we see that the limit law under $P_\Omega\times \leb$ of 
$X_N^a $ is $N(0, \sigma^2(f_a))$. 
Since the law of $X_N^a $ under $P_\Omega\times \leb$ 
is identical with the law under $P_\Omega$, 
we see that $N(0, \sigma^2(f_a))$  is also the limit law under $P$ of 
$X_N^a $. 

Now, denote by $\mathcal G$ the class of integrable functions $g$ on ${\bf R}^d$ 
satisfying
\begin{equation}\label{Eq:cltg}
\lim_{N\to \infty}
\int_{\{X_N^a\le t\}}
g(\boldsymbol x)\,d\boldsymbol x
=
\Phi_{\sigma^2(f)}(t) 
\int_{{\bf R}^d}
g(\boldsymbol x)\,d\boldsymbol x
\quad(t\in {\bf R}), 
\end{equation}
and denote by $\mathcal H$ the collection of 
${\bf 1}_\Omega\in \mathcal G$ 
given by (\ref{Eq:Omega}) using arbitrary $L>0$ and $L^{(h)}_{\delta}\in {\bf R}$ 
($\delta=1$, \dots, $d_h$, $h=1$, \dots, $\beta$). 
We can easily show 
\begin{gather}\label{Eq:linear}
g_1, g_2 \in \mathcal G, \quad \alpha_1, \alpha_2\in{\bf R}
 \implies \alpha_1 g_1 + \alpha_2g_2\in \mathcal G,
\\ \label{Eq:closed}
g_1, g_2, \dots \in \mathcal G,\quad 
\lim_{k\to\infty}\|g-g_k\|_{L^1({\bf R}^d)}=0
 \implies g\in \mathcal G.
\end{gather}
By the above argument we have already proved $\mathcal H \subset \mathcal G$, 
and by (\ref{Eq:linear}) we can see that any  simple function 
which is given as a linear combination of indicator functions with 
supports in  $\mathcal H$ belongs to $\mathcal G$. 
Since any continuous function with compact support 
can be arbitrarily approximated in the sense of $L^1({\bf R}^d)$ by such simple function,  
we see that it belongs to $\mathcal G$. 
Since any  integrable function with compact support 
can be arbitrarily  approximated in the sense of $L^1({\bf R}^d)$ 
by a continuous function with compact support, 
we see that it belongs to $\mathcal G$. 
Hence we can see that (\ref{Eq:clt}) holds under any probability measure $P$ on ${\bf R}^d$ 
which is absolutely continuous with respect to the Lebesgue measure.

In case when $\sigma^2(f_a)=0$, 
by (\ref{Eq:limvar}) we see that  the limit law of $X_N^a$ is the delta measure concentrated on 
$\boldsymbol 0$, that is $N(0, 0)$. 
It proves (\ref{Eq:clt}) for $\sigma^2(f_a)=0$. 
\end{proof}

\begin{proof}[Proof of Theorem \ref{thm:main}]

Put 
$$X_N(\boldsymbol x)= 
\frac1{\sqrt N}
\sum_{k=1}^N f(A^k \boldsymbol x)
\quad\hbox{and}\quad
Y_N^a = X_N-X_N^a
.$$
We have proved
$$
E|Y_N^a|^2 \le \widetilde C \|f-f_a\|_{L^2[0,1)^d}
,$$
which implies
$$
P\bigl(|Y_N^a| \ge  b_a\bigr)
\le \widetilde C b_a
,$$
where $b_a= \|f-f_a\|_{L^2[0,1)^d}^{1/3}$. 
By 
\begin{align*}
P(X_N^a\le t-b_a ) -P(|Y_N^a|\ge b_a) &\le P(X_N\le t) 
\\&
\le P(X_N^a\le t+b_a ) +P(|Y_N^a|\ge b_a)
\end{align*}
we have
\begin{align*}
\Phi_{\sigma^2(f_a)} (t - b_a) - \widetilde C b_a
&\le 
\varliminf_{N\to\infty}P(X_N \le t)
\\&\le 
\varlimsup_{N\to\infty}P(X_N \le t)
\\&\le
\Phi_{\sigma^2(f_a)} (t + b_a) + \widetilde C b_a
.
\end{align*}
By letting $a\to\infty$, we have (\ref{Eq:CLT}). 

By putting 
$$\|X\|_{\widetilde L^2(\Gamma)}= 
\biggl(\frac{1}{\Leb(\Gamma)}\int_\Gamma X^2(\boldsymbol x)\,d\boldsymbol x\biggr)^{1/2},
$$
we have
$$
\|X_N^a\|_{\widetilde L^2(\Gamma)}
-
\|Y_N^a\|_{\widetilde L^2(\Gamma)}
\le
\|X_N\|_{\widetilde L^2(\Gamma)}
\le 
\|X_N^a\|_{\widetilde L^2(\Gamma)}
+
\|Y_N^a\|_{\widetilde L^2(\Gamma)}
$$
and hence  by letting $N\to \infty$, 
\begin{align*}
\sigma(f_a) - ({\widetilde C b_a^3})^{1/2}
&\le \varliminf_{N\to\infty}\|X_N\|_{\widetilde L^2(\Gamma)}
\le \varlimsup_{N\to\infty}\|X_N\|_{\widetilde L^2(\Gamma)}
\le \sigma(f_a) + ({\widetilde C b_a^3})^{1/2}
.
\end{align*}
By letting $a\to\infty$, 
we have (\ref{Eq:limvar}). 
\end{proof}

\section{Acknowledgement}
The author thank the referee for his or her valuable advices, 
especially for the information on the literature \cite{fan0}.

{}

\begin{thebibliography}{}



\bibitem{1981CsR}
M. Cs\"org\H o, P. R\'ev\'esz,
Strong approximations in Probability and Statstics, 
Academic Press,  
1981.


\bibitem{conze}
J. P. Conze, S. Le Borgne, M. Roger, 
{\sl Central limit theorem for stationary products of toral automorphisms}, 
Discrete Cont. Dynm. Sys. 
{\bf 32} (2012) 1597-1626

\bibitem{fan0}
A.-H. Fan, 
{\sl Decay of correlation for expanding toral endomorphims}, 
Dynamical Systems, Proceedings of the international conference in Honor of Professor Liao Shantao, 
Peking University, China, 9-12 August 1998, 
Eds. Y. P. Jiang and L. Wen, World Scientific 1999, 29-40


\bibitem{fan1}
A.-H. Fan, 
{\sl \'Equir\'epartition des orbites d'un endomophisme de ${\bf R}^d$}, 
C. R. Adac. Sci. Paris Ser I, 
{\bf 313} (1991) 735--738

\bibitem{fortet}
R. Fortet, 
Sur une suite egalement r\'epartie, 
Studia Math. 
{\bf 9} (1940) 54--69 


\bibitem{1994F}
K. Fukuyama, 
{\sl The central limit theorem for Riesz-Raikov sums}, 
Probab.  Theory  Related Fields, 
{\bf 100} (1994)  57--75


\bibitem{2015F}
K. Fukuyama \& N. Kuri, 
{\sl The central limit theorem for complex Riesz-Raikov sums}, 
C. R. Math. Acad. Sci. Paris, 
{\bf 353} (2015) 749--753



\bibitem{kac}
M.  Kac,
{\sl  On the distribution of values of sums of type $\sum f(2^kt)$}, 
 Ann.\ Math., 
{\bf 47}  (1946) 
33--49


\bibitem{komlosrevesz}
J. Koml\'os \& P. R\'ev\'esz, 
{\sl Remark to a paper of Gapo\v skin}, 
Acta Sci. Math., 
{\bf 33} (1972) 237--241


\bibitem{leonov}
V. P. Leonov, 
{\sl The central limit theorem for ergodic automorphisms of complact commutative groups}, 
Dokl. Akad. Nauk. SSSR, 
{\bf 133} (1960) 523-526


\bibitem{lesigne}
E. Lesigne, 
Loi des grands nombres pour des sommes de Riesz-Raikov multidimensionnelles 
Compositio Math. 
{\bf 110} (1998) 39-49 


\bibitem{levin}
D. Levin, 
{\sl Central limit theorem for ${\bf Z}_+^d$-actions by toral endomorphism}, 
Electron. J. Probab., 
{\bf 18} (2012) no. 35, 1--42


\bibitem{lobbe}
T. L\"obbe, 
{\sl Limit theorems for multivariate lacunary systems}, 
arXiv: 1408.2202v1 [math.PR] 10 Aug 2014

\bibitem{1991MPh}
D.  Monrad  \&  W. Philipp, 
{\sl Nearby variables with nearby conditional laws 
and a strong approximation theorem for Hilbert space valued martingales}, 
 Probab. Theory Related Fields, 
{\bf 88} (1991) 381--404

\bibitem{moricz}
F. Moricz, 
{\sl Moment inequality and the strong law of the large number}, 
Z. Wahr. verw. Geb., 
{\bf 35} (1976) 299-314

\bibitem{petit}
B.  Petit,
{\sl  Le th\'eor\` eme limite central pour des sommes de Riesz-Raikov}, 
 Probab. Theory Related Fields, 
{\bf  93}  (1992) 
 407--438

\bibitem{baker}
W. Philipp 
{\sl Empirical distribution functions and strong approximation theorems for dependent random variables. A problem of Baker in probabilistic number theory},  
Trans. Amer. Math. Soc.,  
{\bf 345} (1994) 705-727 



\bibitem{1967S}
V. Strassen, 
{\sl Almost sure behavior of sums of independent random variables and martingales},  
Fifth Berkeley Symp. Math. Stat. Prob. Vol II, Part I,
 (1967),  315-343 


\bibitem{shirov}
G. E. Shirov, Linear Algebra, Dover, New York 1977

\bibitem{1962T}
S.  Takahashi,
{\sl On the distribution of values of the type $\sum f(q^k t)$}, 
 T\^ohoku  Math.\ J.,  
{\bf  14}  (1962) 
 233--243

\bibitem{1968Z}
S. K. Zaremba, 
{\sl Some applications of multidimensional integration by parts}, 
Ann. Pol. Math., 
{\bf 21} (1968) 85-96



%

\end{thebibliography}
\end{document}